\documentclass[12pt]{article}
\usepackage{amsfonts}
\usepackage{amssymb,amsmath}
\usepackage{graphicx}
\usepackage{amsthm}
\usepackage{multirow}

\newtheorem{theorem}{Theorem}[section]
\newtheorem{definition}{Definition}[section]
\newtheorem{lemma}{Lemma}[section]
\newtheorem{proposition}{Proposition}[section]

\newtheorem{remark}{Remark}[section]

\newtheorem{algorithm}{Algorithm}[section]

\allowdisplaybreaks

\usepackage{algorithm}
\usepackage{algorithmic}
\usepackage{cite}
\usepackage{multirow}
\usepackage{threeparttable}
\usepackage{slashbox}
\numberwithin{equation}{section}
\usepackage{subfigure}
\begin{document}

\title{An inertial three-operator splitting algorithm with applications to image inpainting}

\author{ Fuying Cui$^{1}${\thanks{email: 381866994@qq.com}}, Yuchao Tang$^{1}${\thanks{Corresponding author. email: hhaaoo1331@163.com}}, Yang Yang$^{1}${\thanks{email: 1179916944@qq.com}}
\\
\\
{ \small 1. Department of Mathematics,} \\
{\small Nanchang University } \\
{\small Nanchang 330031, Jiangxi, P.R. China} \\}

\date{}

%\renewcommand{\theequation}{\thesection.\arabic{equation}}
%\catcode`@=11 \@addtoreset{equation}{section} \catcode`@=12
\date{}
\maketitle{}

 {\bf Abstract.}
The three-operators splitting algorithm is a popular operator splitting method for finding the zeros of the sum of three maximally monotone operators, with one of which is cocoercive operator.
In this paper, we propose a class of inertial three-operator splitting algorithm. The convergence of the proposed algorithm is proved by applying the inertial Krasnoselskii-Mann iteration under certain conditions on the iterative parameters in real Hilbert spaces. As applications,  we develop an inertial three-operator splitting algorithm to solve the convex minimization problem of the sum of three convex functions, where one of them is differentiable with Lipschitz continuous gradient. Finally, we conduct numerical experiments on a constrained image inpainting problem with nuclear norm regularization. Numerical results demonstrate the advantage of the proposed inertial three-operator splitting algorithms.



%%%%%%%%%%%%%%%%%%%%%%%%%%%%%%%%%%%%%%%%%%%%%%%%%%%%%%%%%%%%%%%%%%%%%%%%%%%%%%%%%%%%%%%%%%%%%%%%%%%%%%%%%%%%%%%%%%%%%%%%%
\section{ Introduction}\label{sec-introduction}
\vskip 3mm

Operator splitting algorithms have been widely used for solving many convex optimization problems in signal and image processing, machine learning and
 medical image reconstruction, etc. The traditional operator splitting algorithms include the forward-backward splitting algorithm \cite{lionsandmercier1979}, the Douglas-Rachford splitting algorithm \cite{bouglasandrachford1956TAMS}, and the forward-backward-forward splitting algorithm \cite{Tseng2000SIAM}, which are originally designed for solving the monotone inclusion of the sum of two maximally monotone operators, where one of which is assumed to be cocoercive or just Lipschitz continuous. In recent years, the monotone inclusion problems with the sum of more than two operators have been received much attention. See, for example \cite{Combettes2012SVVA,Combettes2013Systems,vu2013ACM,Vu2015JOTA,Tang2019JCM}.

 Let $H$ be a real Hilbert space. Let $A, B: H\rightarrow 2^H$ be two maximally monotone operators and let $C:H\rightarrow H$ be a cocoercive operator.
 Davis and Yin \cite{davis2015} proposed a so-called three-operator splitting algorithm (also known as Davis-Yin splitting algorithm \cite{Liu2019JOPT}) to solve the monotone inclusion of the form
\begin{equation}\label{more two-monotone-inclusion}
\textrm{ find }\, x\in H \textrm{ such that }\, 0\in Ax+Bx+Cx.
\end{equation}
They pointed out that the three-operator splitting algorithm includes many well-known operator splitting algorithms, such as the forward-backward splitting algorithm \cite{combettes2005}, the Douglas-Rachford splitting algorithm \cite{combettes2007} and the backward-forward splitting algorithm.
Raguet et al. \cite{Raguet-SIAM-2013} proposed a generalized forward-backward splitting algorithm to solve the monotone inclusion of
\begin{equation}\label{more-monotone-inclusion}
\textrm{ find }\, x\in H \textrm{ such that }\, 0\in{ Bx+\sum_{i=1}^{m}A_{i}x},
\end{equation}
where $m\geq 1$ is an integer,  $\{A_{i}\}_{i=1}^{m}: H\rightarrow 2^H$ are maximally monotone operators and $B:H\rightarrow H$ is a cocoercive operator. In particular, when $m=1$, the generalized forward-backward splitting algorithm reduces to the forward-backward splitting algorithm.
Furthermore, Raguet and Landrceu \cite{Raduet2015SIAMJIS} presented a precondition of the generalized forward-backward splitting algorithm for solving the monotone inclusion (\ref{more-monotone-inclusion}).
By introducing a suitable product space, the monotone inclusion (\ref{more-monotone-inclusion}) can be reformulated as the sum of three maximally monotone operators, where one of them is a normal cone of a closed vector subspace. To solve this equivalent monotone inclusion problem, Briceno-Arias \cite{briceno2015Optim} introduced two splitting algorithms: the Forward-Douglas-Rachford splitting algorithm and the forward-partial inverse splitting algorithm. As a consequence, the generalized forward-backward splitting algorithm could be derived from the Forward-Douglas-Rachford splitting algorithm. Notice that the normal cone of a closed vector subspace is also a maximally monotone operator, then the three-operator splitting algorithm can be applied to solve the monotone inclusion studied in \cite{briceno2015Optim}, which also results in the generalized forward-backward splitting algorithm \cite{Raguet-SIAM-2013}. Some recent generalization of the three-operator splitting algorithm in the direction of stochastic and inexact can be found in \cite{Cevher2016Report,Zong2018}.

In recent years, the inertial method has become more and more popular. Various inertial type algorithms were studied, see for example \cite{Alvarez2004,Ochs2014,Bot2015NFAO,Chen2015,Dong2016Optim,Combettes2017,Attouch2018} and references therein. The inertial method is also called the heavy ball method, which is based on a discretization of a second order dissipative dynamic system. In 1983, Nesterov \cite{nesterov1983} proposed an inertial gradient descent algorithm, which modified the heavy ball method of Polyak \cite{polyak1964}. Further, G\"{u}ler \cite{Guler1992} generalized Nesterov's method to the proximal point algorithm for solving the minimization of a nonsmooth convex function. In \cite{Alvarez2001}, Alvarez and Attouch proposed an inertial proximal point algorithm (iPPA) for solving zeros of a maximally monotone operator, which is defined by the following: let $x^{0}, x^{-1}\in H$, and set
\begin{equation}\label{inertial-proximal-point-algorithm}
\left\{
\begin{aligned}
\omega^{k} &=x^{k}+\alpha_{k}(x^{k}-x^{k-1}), \\
x^{k+1} &=(I+\lambda_{k}A)^{-1}(\omega^{k}).
\end{aligned}
\right.
\end{equation}
They proved the convergence of (\ref{inertial-proximal-point-algorithm}) under the condition: $\inf \lambda_k >0$ and
\begin{equation}\label{inertial-parameter-1}
\begin{aligned}
& \textrm{ (i) }\, \{\alpha_k \} \subseteq [0,\overline{\alpha}),\, \overline{\alpha}\in [0,1), \\
& \textrm{ (ii) } \sum_{k=0}^{+\infty} \alpha_k \| x^k - x^{k-1} \|^2 < +\infty.
\end{aligned}
\end{equation}
Moudafi and Oliny \cite{Moudafi2003} introduced an inertial algorithm for finding zeros of the sum of two maximally monotone operators $A$ and $B$, where $B$ is $\beta$-cocoercive, for some $\beta >0$.
The following iterative algorithm is defined in \cite{Moudafi2003}.
\begin{equation}\label{inertial-algorithm-MO}
\left\{
\begin{aligned}
\omega^{k} &=x^{k}+\alpha_{k}(x^{k}-x^{k-1}), \\
x^{k+1} &=(I+\lambda_{k}A)^{-1}(\omega^{k} - \lambda_k Bx^k).
\end{aligned}
\right.
\end{equation}
They proved the convergence of the proposed inertial algorithm (\ref{inertial-algorithm-MO}) under the same condition (\ref{inertial-parameter-1}) imposed on the inertial parameters $\{\alpha_k\}$ as well as $\lambda_k \in (0, 2\beta)$. Lorenz and Pock \cite{lorenz2015JMIV} proposed a preconditioner, inertial forward-backward splitting algorithm as follows, in which the cocoercive operator is evaluated at the inertial extrapolate iteration scheme and a symmetric, positive define linear operator is used as a preconditioner.
\begin{equation}\label{inertial-algorithm-LP}
\left\{
\begin{aligned}
\omega^{k} &=x^{k}+\alpha_{k}(x^{k}-x^{k-1}), \\
x^{k+1} &=(I+\lambda_{k}M^{-1}A)^{-1}(\omega^{k} - \lambda_k M^{-1} B\omega^k),
\end{aligned}
\right.
\end{equation}
where $M$ is a linear self-adjoint and positive define operator.
Comparing (\ref{inertial-algorithm-MO}) and (\ref{inertial-algorithm-LP}), we can see that the operator $B$ is evaluated at the current point $x^k$ in (\ref{inertial-algorithm-MO}), but it is calculated at the inertial term in (\ref{inertial-algorithm-LP}).

The inertial proximal point algorithm (\ref{inertial-proximal-point-algorithm}) has been proved an effective way to accelerate the speed of the proximal point algorithm. Before the inertial proximal point algorithm, the relaxed proximal point algorithm proposed by Eckstein and Bertsekas \cite{Eckstein1992} is another approach for accelerating the proximal point algorithm. For this purpose, Maing\'{e} \cite{Mainge2008JCAM} combined the relaxed strategy with the inertial method, and proposed the following relaxed inertial proximal point algorithm
 \begin{equation}\label{relaxed-inertial-proximal-point-algorithm}
\left\{
\begin{aligned}
\omega^{k} &=x^{k}+\alpha_{k}(x^{k}-x^{k-1}), \\
x^{k+1} &= (1-\rho_k)\omega^k + \rho_k(I+\lambda_{k}A)^{-1}(\omega^{k}).
\end{aligned}
\right.
\end{equation}
Under the condition $0<\inf \rho_k \leq \sup \rho_k <2$ and (\ref{inertial-parameter-1}), the weak convergence of (\ref{relaxed-inertial-proximal-point-algorithm}) was proved in \cite{Mainge2008JCAM}, which was based on the inertial Krasnoselskii-Mann iteration (\ref{inertial-KM-iteration}).

In order to ensure the convergence of the inertial algorithms mentioned above, the condition (\ref{inertial-parameter-1}) is usually enforced. To deal with the issue of choosing the inertial parameters $\{\alpha_k\}$ as a priori, Alvarez and Attouch \cite{Alvarez2001} proved the convergence of the inertial proximal point algorithm (\ref{inertial-proximal-point-algorithm}) under the requirement of $\{\alpha_k\}$ is a nondecreasing sequence in $[0,\overline{\alpha})$ with $\overline{\alpha}< 1/3$. Lorenz and Pock \cite{lorenz2015JMIV} also proved the convergence of the inertial forward-backward splitting algorithm (\ref{inertial-algorithm-LP}) without the condition (\ref{inertial-parameter-1}). Bo\c{t} et al. \cite{Bot2015AMC} proved the convergence of the inertial Krasnoselskii-Mann iteration (\ref{inertial-KM-iteration}) without the condition (\ref{inertial-parameter-1}). Consequently, they proposed an inertial Douglas-Rachford splitting algorithm for solving the monotone inclusion problem of the sum of two maximally monotone operators. Further, in \cite{Bot2016MTA}, Bo\c{t} and Csetnek proposed an inertial alternating direction method of multipliers based on the inertial Douglas-Rachford splitting algorithm.
In the context of convex minimization, the inertial forward-backward splitting algorithm leads to the so-called inertial proximal gradient algorithm. Just we have mentioned that the inertial parameters affect greatly the performance of the inertial algorithms. The Nesterov's sequence is one of popular choice, which reduces to the famous iterative shrinkage thresholding algorithm (FISTA) \cite{beck2009}. However, the convergence of the sequences generated by FISTA has been missed for a long time. Recently, Chambolle and Dossal \cite{Chambolle2015} proposed a different kind of inertial parameters for studying the convergence of the inertial proximal gradient algorithm. The convergence rate of the inertial proximal gradient algorithm with various options for the inertial sequences $\{\alpha_k\}$ were recently established in \cite{Attouch2016SJO, Attouch2018SJO}.  An important property of these special inertial parameters for the inertial proximal gradient algorithm is that they all converge to one. However, it imposes more restrictions on the step size and also excludes the relaxation parameters. In addition, the convergence properties of the inertial proximal gradient algorithm with special choices of inertial parameters is limited in the context of convex minimization. It is not cleared whether these results can be extended to the general inertial forward-backward splitting algorithm for solving the monotone inclusion problem, such as (\ref{inertial-algorithm-MO}) or (\ref{inertial-algorithm-LP}).

The purpose of this paper is to study a class of inertial three-operator splitting algorithm for solving the monotone inclusion problem (\ref{more two-monotone-inclusion}), which unifies the three-operator splitting algorithm and the inertial methods. We analyze the convergence of the proposed iterative algorithm under different conditions on the parameters. As a consequence, we obtain an inertial three-operator splitting algorithm for solving the convex minimization problem of the sum of three convex functions, where one of them is differentiable with Lipschitz continuous gradient and the others are proximable friendly convex functions. We verify the advantage of the proposed inertial three-operator splitting algorithm by applying it to a constrained image inpainting problem.

The rest of this paper is organized as follows. Section 2, we recall some notations and definitions in monotone operators theory and convex analysis. Moreover, we also give some technical lemmas, which will be used in the following sections. Section 3, we propose a class of inertial three-operator splitting algorithm and establish the main convergence theorems. We also apply the proposed iterative algorithm to solve the convex optimization problem with the sum of three convex functions. Section 4, we apply the proposed inertial three-operator splitting algorithm to solve a constrained image inpainting problem and report numerical experiments results. Finally, we give some conclusions and future works.

%%%%%%%%%%%%%%%%%%%%%%%%%%%%%%%%%%%%%%%%%%%%%%%%%%%%%%%%%%%%%%%%%%%%%%%%%%%%%%%%%%%%%%%%%%%%%%%%%%%%%%%%%%%%%%%%%%%%%%%%%%%%%%%%%%%%%%%%%%%%%
\section{Preliminaries}\label{sec:pre}
\vskip 3mm

In this section, we review some  basic definitions and lemmas in monotone operator theory and convex analysis. Let $H$ be a real Hilbert space, the scalar product in $H$ is denoted by $\langle\cdot,\cdot\rangle$ and the corresponding norm in $H$ is $\|\cdot\|$.
Let $A:H\rightarrow 2^H$ be a set-valued operator. We denote its domain, range, graph and zeros  by dom $A= \{ x\in H| Ax \neq \emptyset \}$, ran $A = \{ u\in H | (\exists x\in H) u\in Ax\}$,  gra $A = \{ (x,u)\in H\times H | u\in Ax \}$, and zer $A = \{x\in H  | 0\in Ax\}$, respectively. $A^{-1}$ be inverse operator of $A$, defined by $(x,u)\in \textrm{gra } A$ if and only if $(u,x)\in \textrm{gra } A^{-1}$.

\begin{definition}(\cite{bauschkebook2017})
Let $A:H\rightarrow 2^H$ be a set-valued operator.

\noindent \emph{(i)}\, $A$ is said to be monotone, if
$$
\langle x-y,u-v \rangle \geq 0, \quad \forall (x,u), (y,v)\in \textrm{gra } A.
$$
Moreover, $A$ is said to be maximally monotone, if its graph is not strictly contained in the graph of any other monotone operator.

\noindent \emph{(ii)}\, $A$ is said to be uniformly monotone, if there exists an increasing function
$\psi: [0,+\infty) \rightarrow [0,+\infty]$ which vanishes only at $0$ such that
$$
\langle x-y,u-v \rangle \geq \psi(\|x-y\|),  \forall (x,u), (y,v)\in \textrm{gra } A.
$$
If $\psi =\gamma (\cdot)^{2}$, then the $A$ is called $\gamma$-strongly monotone.

\noindent \emph{(iii)}\, $A$ is said to be demiregular at $x \in \textrm{dom}\, A$, if for all $u\in Ax$ and for all sequences $(x^k, u^k)\in gra\, A$ with $x^k \rightharpoonup x$ and $u^k \rightarrow u$, we have $x^k\rightarrow x$.
\end{definition}

\begin{definition}(\cite{bauschkebook2017})
Let $B:H\rightarrow H$ be a single-valued operator. $B$ is called $\beta$-cocoercive, where $\beta \in (0, +\infty)$, if
$$
\langle x-y, Bx-By \rangle \geq \beta \|Bx-By\|^{2}, \quad \forall x,y\in H.
$$
\end{definition}

\begin{definition}(\cite{bauschkebook2017})
Let $A:H\rightarrow 2^H$ be a maximally monotone operator. The resolvent operator of $A$ with index $\gamma >0$ is defined as
$$
J_{\gamma A} = (I+\gamma A)^{-1}.
$$
where $I$ is the identity operator.
\end{definition}

Next, we recall definitions of nonexpansive and related nonlinear operators. These operators often appear in the convergence analysis of operator splitting algorithms.

\begin{definition}(\cite{bauschkebook2017})
Let $C$ be a nonempty subset of $H$. Let $T:C\rightarrow H$, then

\noindent \emph{(i)} $T$ is called nonexpansive, if
$$
\|Tu-Tv\| \leq \|u-v\|, \quad \forall u,v\in C.
$$

\noindent \emph{(ii)} $T$ is called firmly nonexpansive, if
$$
\|Tu-Tv\|^2 \leq \|u-v\|^2 - \| (I-T)u- (I-T)v \|^2, \quad \forall u,v\in C.
$$

\noindent \emph{(iii)} $T$ is called $\theta$-averaged, where $\theta\in (0,1)$, if there exists a nonexpansive operator $S$ such that $T = (1-\theta)I + \theta S$.
\end{definition}

It is easy to prove that $T$ is $\theta$-averaged if and only if
$$
\|Tu-Tv\|^2 \leq \|u-v\|^2 - \frac{1-\theta}{\theta}\|(I-T)u-(I-T)v\|, \quad \forall u,v\in C.
$$.
On the other hand, the resolvent operator $J_{\gamma A}$ is firmly nonexpansive and also nonexpansive operator.

Davis and Yin \cite{davis2015} proved the following important lemma, which provides a fixed point characterize of the three-operator monotone inclusion problem (\ref{more two-monotone-inclusion}).

\begin{lemma}(\cite{davis2015})\label{three-operator-lemma}
Let $\gamma >0$, define $T = J_{\gamma A}(2J_{\gamma B} - I - \gamma CJ_{\gamma B}) + I - J_{\gamma B}$. The following set equality holds
\begin{equation}
zer(A+B+C)=J_{\gamma B}(Fix (T)). \nonumber
\end{equation}
In addition
\begin{equation}
Fix (T) = \{x+\gamma \mu | 0 \in (A+B+C)x, \mu \in (Bx) \bigcap (-Ax-Cx)\}. \nonumber
\end{equation}
\end{lemma}

\begin{proposition}(\cite{davis2015})\label{three-operator-prop}
Suppose that $T_{1}, T_{2}:H\rightarrow H$ are firmly nonexpansive and $C$ is $\beta$-cocoercive, for some $\beta > 0 $. Let $\gamma \in (0,2\beta)$. Then
\begin{equation}
T:=I-T_{2}+T_{1}\circ (2T_{2}-I-\gamma C\circ T_{2}),
\end{equation}
 is $\alpha$-averaged,  with $\alpha =\frac {2\beta}{4\beta-\gamma}< 1$.
In particular, the following inequality holds
\begin{equation}
\|Tx-Ty\|^2 \leq \|x-y\|^2 - \frac{1-\alpha}{\alpha}\|(I-T)x-(I-T)y\|, \quad \forall x,y\in H.
\end{equation}
Further, for any $\overline{\varepsilon} \in (0,1)$ and $\gamma \in (0,2\beta\overline{\varepsilon})$,
let $\overline{\alpha} =\frac {1} {2-\overline{\varepsilon}} < 1$. Then the following holds for all $x,y \in H$
\begin{align}
\|Tx-Ty\|^2 &\leq \|x-y\|^2 - \frac{1-\overline{\alpha}}{\overline{\alpha}}\|(I-T)x-(I-T)y\|^{2}\nonumber \\
            &-\gamma(2\beta-\frac {\gamma}{\overline{\varepsilon}})\|C\circ T_{2}x-C\circ T_{2}y\|^{2},
 \quad \forall x,y\in H.
 \end{align}

\end{proposition}

The Krasnoselskii-Mann (KM) iteration scheme plays an important role in studying various fixed point algorithms arising in signal and image processing. Let $x^0\in H$, the KM iteration is defined as follow
\begin{equation}\label{KM-iteration}
x^{k+1} = (1-\rho_k)x^k + \rho_k Tx^k, \quad  k \geq 0,
\end{equation}
where $\rho_k \in (0,1)$. In order to accelerate the KM iteration (\ref{KM-iteration}), the inertial Krasnoselskii-Mann (iKM) iteration was introduced. Let $x^{0}, x^{-1}\in H$, the iKM iteration scheme reads as
\begin{equation}\label{inertial-KM-iteration}
\left \{
\begin{aligned}
y^k & = x^k + \alpha_k (x^k - x^{k-1}), \\
x^{k+1} & = (1-\rho_k)y^k + \rho_k Ty^k,\quad k \geq 0,
\end{aligned}\right.
\end{equation}
where $\rho_k \in (0,1)$ and $\alpha_k \in (0,1)$. The following results concerned with the convergence analysis of the iKM iteration (\ref{inertial-KM-iteration}).

Maing\'{e} \cite{Mainge2008JCAM} proved the following convergence result of the iKM iteration (\ref{inertial-KM-iteration}).

\begin{lemma}(\cite{Mainge2008JCAM})\label{iKM-convergence-1}
Let $H$ be a real Hilbert space. Let $T:H\rightarrow H$ be a nonexpansive operator such that $Fix (T)\neq \varnothing $. Let $\{x^k\}$ be generated by (\ref{inertial-KM-iteration}), where $\{\rho_k\}\subset (0,1)$ and $\{\alpha_k\}\subset [0,1)$ satisfy the following conditions:

\noindent \emph{(i)} $0 \leq \alpha_{k} \leq \alpha < 1 $, $ 0<  \underline{\lambda} \leq \lambda_{k} \leq \overline{\lambda} <1$.

\noindent \emph{(ii)} $\sum_{k=0}^{+\infty}\alpha_{k}\|x^{k}-x^{k-1}\|^{2} < +\infty$.

\noindent Then

\noindent \emph{(a)} For any $x^{*} \in Fix(T), \lim_{k\rightarrow +\infty}\|x^{k}-x^{*}\|$ exists;

\noindent \emph{(b)} $\{x^{k}\}$ converges weakly to a point in $Fix(T)$.

\end{lemma}

Bo\c{t} et al. \cite{Bot2015AMC} removed the condition (ii) in Lemma \ref{iKM-convergence-1}, but needed stronger conditions on $\{\alpha_k\}$ and $\{\lambda_k\}$. They proved the following convergence result.

\begin{lemma}(\cite{Bot2015AMC})\label{iKM-convergence-2}
Let $H$ be a real Hilbert space. Let $T:H\rightarrow H$ be a nonexpansive operator such that $Fix (T)\neq \varnothing $. Let $\{x^k\}$ be generated by (\ref{inertial-KM-iteration}), where $\{\rho_k\}\subset (0,1)$ and $\{\alpha_k\}\subset [0,1)$ satisfy the following conditions:

\noindent \emph{(i)} $\{\alpha_{k}\}_{k\geq 1}$ is nondecreasing  with $\alpha_{1}=0$ and $0\leq \alpha_{k} \leq \alpha <1$
for every $k\geq 1$;

\noindent \emph{(ii)} Let $\lambda, \sigma, \delta > 0 $ such that
\begin{equation}
\delta > \frac{\alpha^{2}(1+\alpha)+\alpha \sigma}{1-\alpha^{2}} \quad and \quad0<\lambda \leq \lambda_{k}\leq \frac{\delta-\alpha[\alpha (1+\alpha)+\alpha \delta+\sigma]}{\delta[1+\alpha (1+\alpha)+\alpha \delta+\sigma]}
\end{equation}

\noindent Then the following hold:

\noindent \emph{(a)} For any $y \in Fix(T), \lim_{k\rightarrow +\infty}\|x^{k}-y\|$ exists; \\
\noindent \emph{(b)} $\sum _{k=0}^{\infty}\|x^{k+1}-x^{k}\|^{2} < +\infty$; \\
\noindent \emph{(c)} $\{x^{k}\}$ converges weakly to a point in $Fix(T)$.

\end{lemma}

\begin{figure}[htbp]
     \setlength{\abovecaptionskip}{-5pt}
  \centering \scalebox{0.6}{\includegraphics{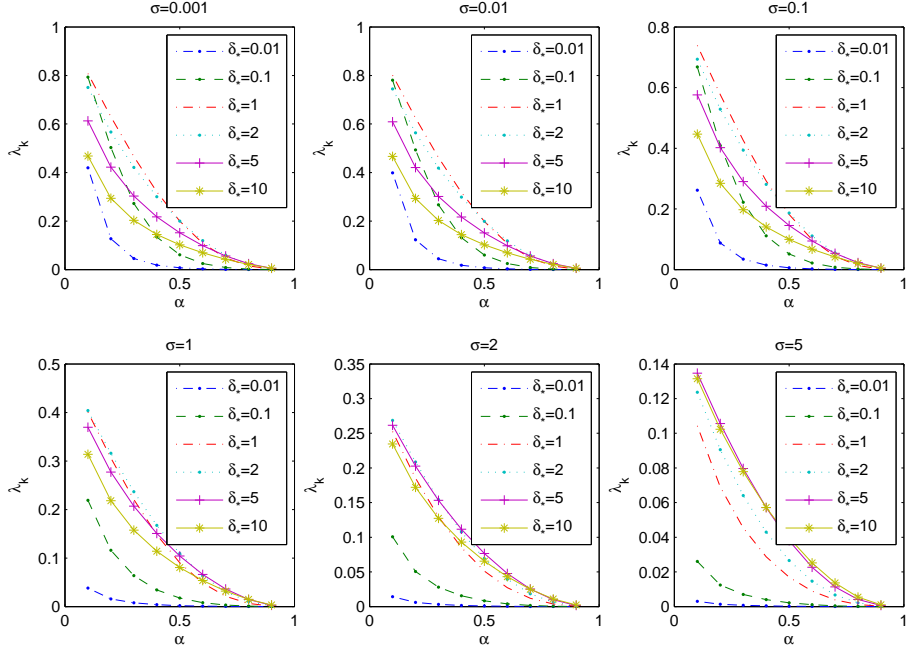}}
  \caption[]{The relaxation parameters versus the given constants of  $\delta_{*}, \sigma$ and $\alpha$.}%
{\label{parameter-range}}
\end{figure}

\begin{remark}
The inertial parameters $\{\alpha_k\}$ in Lemma \ref{iKM-convergence-1} have to be calculated based on the current iteration and the last iteration. In Lemma \ref{iKM-convergence-2}, the relaxation parameters $\{\lambda_k\}$ are restricted by the constants $\delta, \sigma$ and $\alpha$. To show the impact of these parameters to the relaxation parameters, we plot the range of the relaxation parameters $\{\lambda_k\}$ in Figure \ref{parameter-range}. Here, we let $\delta = \frac{\alpha^{2}(1+\alpha)+\alpha \sigma}{1-\alpha^{2}} + \delta_{*}$. To get a large selection of the relaxation parameters $\{\lambda_k\}$, we recommend to set $\sigma = 0.01$ or $\sigma = 0.001$ and $\delta_{*}=1$. From Figure \ref{parameter-range}, it is obvious that the relaxation parameters $\{\lambda_k\}$ are decreasing as the maximal inertial parameter $\alpha$ increasing.
\end{remark}

Finally, we close this section by introducing some elements in convex analysis. See \cite{Zalinescu2012}. Let $f:H\rightarrow (-\infty, +\infty]$. We denote that $\textrm{ gra } f = \{x\in H | f(x)< +\infty\}$ is effective domain of $f$. We call $f$ is proper if $\textrm{ gra } f \neq \emptyset$. Let $\Gamma_{0}(H)$ be the
class of proper lower semicontinuous convex functions from $H$ to $(-\infty, + \infty]$. The subdifferential of $f:H\rightarrow (-\infty, +\infty]$ is the set $\partial f(x) = \{u\in H | f(y)\geq f(x) + \langle u, y-x\rangle, \forall y\in H\}$.

\begin{definition}(\cite{Zalinescu2012})
Let $f\in \Gamma_{0}(H)$. $f$ is called uniformly convex, if there exists an increasing function $\phi: [0,+\infty) \rightarrow [0,+\infty]$ that vanishes only at $0$ such that, for every $\alpha\in (0,1)$ and every $x,y\in \textrm{ dom } f$,
$$
f(\alpha x + (1-\alpha)y) \leq \alpha f(x) + (1-\alpha)f(y) - \alpha (1-\alpha)\phi (\|x-y\|).
$$
\end{definition}

\begin{definition}(\cite{Zalinescu2012})
Let $f\in \Gamma_{0}(H)$. Let $u\in H$. The proximity operator of $f$ with index $\lambda >0$ is defined by
$$
prox_{\lambda f}(u) = \arg\min_{x} \{ \frac{1}{2\lambda} \|x-u\|^2 + f(x) \}.
$$
\end{definition}
Notice that $J_{\lambda \partial f} = prox_{\lambda f}$. The proximity operator is a generalization of the projection operator $P_{C}$, where $C$ is a nonempty closed convex set.

%%%%%%%%%%%%%%%%%%%%%%%%%%%%%%%%%%%%%%%%%%%%%%%%%%%%%%%%%%%%%%%%%%%%%%%%%%%%%%%%%%%%%%%%%%%%%%%%
\section{An inertial three-operator splitting algorithm}
\vskip 3mm

In this section, we propose an inertial three-operator splitting algorithm to solve the monotone inclusion problem (\ref{more two-monotone-inclusion}). This algorithm combines the inertial iterative algorithm (\ref{inertial-KM-iteration}) with the three-operator splitting algorithm \cite{davis2015}. The detailed iterative algorithm is presented in Algorithm \ref{alg1}.

\renewcommand{\algorithmicrequire}{\textbf{Input:}}
\renewcommand{\algorithmicensure}{\textbf{Output:}}
\begin{algorithm}[htb]
\caption{ An inertial three-operator splitting (iTOS) algorithm}
\label{alg1}
\begin{algorithmic}[1]
\REQUIRE
For any given $z^{0}$, $z^{-1}\in H$. Choose $\gamma, \lambda_{k}$ and $\alpha_{k}$. \\

For $k=0,1,2,\cdots$, do
\quad  \STATE $y^{k}=z^{k}+\alpha_{k}(z^{k}-z^{k-1})$;
\quad  \STATE $x_{B}^{k}=J_{\gamma B}y^{k}$;
\quad \STATE $x_{A}^{k}=J_{\gamma A}(2x_{B}^{k}-y^{k}-\gamma Cx_{B}^{k})$;
\quad \STATE $z^{k+1}=y^{k}+\lambda _{k}(x_{A}^{k}-x_{B}^{k})$.

End for when some stopping criterion is met.
\end{algorithmic}
\end{algorithm}

Let $\alpha_k =0$, then the inertial three-operator splitting algorithm reduces to the original three-operator splitting algorithm introduced by Davis and Yin \cite{davis2015}. Notice that most of the operator splitting algorithms can be written as KM iteration scheme (\ref{KM-iteration}) for computing fixed points of nonexpansive operators. The iKM iteration (\ref{inertial-KM-iteration}) is useful for developing various inertial operator splitting algorithms. Therefore, we shall make full use of Lemma \ref{iKM-convergence-1} and Lemma \ref{iKM-convergence-2} to prove the convergence of Algorithm \ref{alg1}. Now, we are ready to prove the first convergence theorem of Algorithm \ref{alg1}.

\begin{theorem}\label{main-theorem}
Let $H$ be a real Hilbert space. Let $A:H \rightarrow 2^{H}$ and $B:H\rightarrow 2^H$ be two maximally monotone operators. Let $C:H\rightarrow H$ be a $\beta$-cocoercive operator, for some $\beta> 0$. Define an operator $T:H\rightarrow H$ as follows
\begin{equation}\label{def-T}
T:= I-J_{\gamma B}+J_{\gamma A}(2J_{\gamma B}-I-\gamma C J_{\gamma B}).
\end{equation}
Let the iterative sequences $\{z^{k}\}, \{x_{A}^{k}\}$, and $\{x_{B}^{k}\}$ are generated by Algorithm \ref{alg1}. Assume that the parameters $\gamma, \{\alpha_{k}\}$, and $\{\lambda_{k}\}$ satisfy the following conditions:

\noindent \emph{(i1)} $\gamma \in (0, 2\beta \overline{\varepsilon})$, where $\overline{\varepsilon} \in (0,1)$.

\noindent \emph{(i2)} $\{\alpha_{k}\}$ is nondecreasing with $k\geq 1, \alpha_{1}=0$ and $ 0\leq \alpha_{k} \leq \alpha < 1 $.

\noindent \emph{(i3)}  for every $k\geq 1$, and $\lambda, \sigma, \delta > 0 $ such that
\begin{equation}
\delta > \frac{\alpha^{2}(1+\alpha)+\alpha \sigma}{1-\alpha^{2}} \textrm{ and }\, 0<\lambda \leq \lambda_{k}\leq \frac{\delta-\alpha[\alpha (1+\alpha)+\alpha \delta+\sigma]}{\overline{\alpha}\delta[1+\alpha (1+\alpha)+\alpha \delta+\sigma]},
\end{equation}
where $\bar{\alpha} = \frac{1}{2-\overline{\varepsilon}}$. Then the following hold

\noindent \emph{(i)} $\{z^{k}\}$ converges weakly to a fixed point of $T$.

\noindent \emph{(ii)} Let $\lambda_{k}\geq \underline{\lambda}> 0 $ and $z^{*}$ be a fixed point of $T$. Then $\{x_{B}^{k}\}$ converges weakly to $ J_{\gamma B}z^{*}\in zer(A+B+C)$.

\noindent \emph{(iii)} Let $\lambda_{k}\geq \underline{\lambda}> 0 $ and $z^{*}$ be a fixed point of $T$. Then $\{x_{A}^{k}\}$ converges weakly to $ J_{\gamma B}z^{*}\in zer(A+B+C)$.

\noindent \emph{(iv)} Let $\lambda_{k}\geq \underline{\lambda}> 0 $ and $z^{*}$ be a fixed point of $T$. Suppose that one of the following conditions hold

\noindent \emph{(a)} $A$ is uniformly monotone on every nonempty bounded subset of $\textrm{ dom } A$.\\
\noindent \emph{(b)} $B$ is uniformly monotone on every nonempty bounded subset of $\textrm{ dom } B$.\\
\noindent \emph{(c)} $C$ is demiregular at every point $x \in zer(A+B+C)$. \\
Then $\{x_{A}^{k}\}$ and $\{x_{B}^{k}\}$ converge strongly to $J_{\gamma B}z^{*}\in zer(A+B+C)$.
\end{theorem}

\begin{proof}
The iterative sequence $\{z^k\}$ generated by Algorithm \ref{alg1} is equivalent to
\begin{equation}\label{eq3-1}
z^{k+1} =(1- \lambda_{k})y^{k}+\lambda_{k}Ty^{k}.
\end{equation}
It follows from Proposition \ref{three-operator-prop} that $T$ is $\overline{\alpha}$-averaged. Then there exists a nonexpansive operator $R$ such that $T=(1-\overline{\alpha})I+\overline{\alpha} R$.
Hence, we obtain from (\ref{eq3-1}) that
\begin{align}\label{eq3-2}
z^{k+1} & = (1- \lambda_{k})y^{k}+\lambda_{k}((1-\overline{\alpha})y^{k}+ \overline{\alpha} Ry^{k})\nonumber \\
       &  = (1- \lambda_{k}\overline{\alpha})y^{k}+\lambda_{k}\overline{\alpha} Ry^{k}.
\end{align}

Notice that $Fix(T)=Fix(R)\neq \emptyset$, and the conditions of (i1) and (i2) imply that all the conditions of Lemma \ref{iKM-convergence-1} are satisfied. Then we obtain that
for any $y \in Fix(R)$, $\lim_{k\rightarrow +\infty}\|z^{k}-y\|$ exists. Moreover, $\sum _{k=0}^{\infty}\|z^{k+1}-z^{k}\|^{2} < +\infty$ and  $\{z^{k}\}$ converges weakly to a point in $Fix(R)$.

\noindent (i)\, Since $Fix(T)=Fix(R)$, then we obtain the conclusion of (i).

\noindent (ii)\, Let $z^{*} \in Fix(T)$, from (\ref{eq3-1}), we have
\begin{align}\label{eq3-3}
\|z^{k+1}-z^{*}\|^{2} &=\|(1- \lambda_{k})(y^{k}-z^{*})+\lambda_{k}(Ty^{k}-z^{*})\|^{2}\nonumber \\
                      &=(1- \lambda_{k})\|y^{k}-z^{*}\|^{2}+\lambda_{k}\|Ty^{k}-z^{*}\|^{2}-\lambda_{k}(1- \lambda_{k})\|y^{k}-Ty^{k}\|^{2}.
\end{align}
Let $T_1 = J_{\gamma A}$ and $T_2 = J_{\gamma B}$ in Proposition \ref{three-operator-prop}, then we have
\begin{align}\label{eq3-4}
\|Ty^{k}-z^{*}\|^{2} &\leq \|y^{k}-z^{*}\|^{2}-\frac{(1-\overline{\alpha})}{\overline{\alpha}}\|(I-T)y^{k}-(I-T)z^{*}\|^{2}\nonumber\\
 &-\gamma (2\beta-\frac {\gamma}{\overline{\epsilon}})\|CJ_{\gamma B}(y^{k})- CJ_{\gamma B}(z^{*})\|^{2}.
\end{align}
Substituting (\ref{eq3-4}) into (\ref{eq3-3}), we obtain
\begin{align}\label{eq3-5}
\|z^{k+1}-z^{*}\|^{2} &\leq \|y^{k}-z^{*}\|^{2}- \lambda_{k}(\frac{1}{\overline{\alpha}}-\lambda_{k})\|Ty^{k}-y^{k}\|^{2}\nonumber\\
 &-\lambda_{k}\gamma (2\beta-\frac {\gamma}{\epsilon})\|CJ_{\gamma B}(y^{k})- CJ_{\gamma B}(z^{*})\|^{2}.
\end{align}
Next, we prove $\lim_{k\rightarrow +\infty}\|y^{k}-z^{*}\|=\lim_{k\rightarrow +\infty}\|z^{k}-z^{*}\|.$ In fact,
\begin{align}\label{eq3-6}
\|y^{k}-z^{*}\| &=\|z^{k}+\alpha_{k}(z^{k}-z^{k-1})-z^{*} \nonumber\|\\
                &\leq\|z^{k}-z^{*}\|+\alpha_{k}\|z^{k}-z^{k-1}\|.
\end{align}
On the other hand,
\begin{align}\label{eq3-7}
\|z^{k+1}-z^{*}\| &=\|(1-\lambda_{k})(y^{k}-z^{*})+\lambda_{k}(Ty^{k}-z^{*})\| \nonumber\\
                &\leq\|y^{k}-z^{*}\|.
\end{align}
Observe that $\lim_{k\rightarrow +\infty}\|z^{k}-z^{k-1}\|=0$ and $\lim_{k\rightarrow +\infty}\|z^{k}-z^{*}\|$ exists.
We can conclude from (\ref{eq3-6}) and (\ref{eq3-7}) that $\lim_{k\rightarrow +\infty}\|y^{k}-z^{*}\|=\lim_{k\rightarrow +\infty}\|z^{k}-z^{*}\|$.
Therefore, from (\ref{eq3-5}), we obtain
\begin{equation}\label{eq3-8}
\lim_{k\rightarrow +\infty}\|Ty^{k}-y^{k}\|=0,
\end{equation}
and
\begin{equation}\label{eq3-9}
\lim_{k\rightarrow +\infty}\|CJ_{\gamma B}(y^{k})- CJ_{\gamma B}(z^{*})\|=0.
\end{equation}

Let $\mu_{B}^{k}=\frac {1}{\gamma}(y^{k}-x_{B}^{k})\in Bx_{B}^{k}$. $\mu_{A}^{k}=\frac {1}{\gamma}(2x_{B}^{k}-y^{k}-\gamma C x_{B}^{k}-x_{A}^{k}
)\in Ax_{A}^{k}$. It follows from the nonexpansiveness of $J_{\gamma B}$, we have
\begin{align}\label{eq3-10}
\|x_{B}^{k}-J_{\gamma B}(z^{*})\| &\leq\|y^{k}-z^{*}\| \nonumber \\
                &\leq\|z^{k}-z^{*}\|+\alpha_{k}\|z^{k}-z^{k-1}\|\nonumber\\
                &\leq\|z^{k}-z^{*}\|+\alpha\|z^{k}-z^{k-1}\|.
\end{align}
Notice that
$\lim_{k\rightarrow +\infty}\|z^{k}-z^{*}\|$ exists and $\lim_{k\rightarrow +\infty}\|z^{k}-z^{k-1}\|=0$, then $\{x_{B}^{k}\}$ is bounded.
Let $y$ is a sequential weak cluster point of $\{x_{B}^{k}\}$. That is, there exists a subsequence $\{x_{B}^{k_{n}}\}$ such that $x_{B}^{k_{n}}\rightharpoonup y$ as $k_{n}\rightarrow +\infty$.
Let $x^{*}=J_{\gamma B}(z^{*})$. Then $x^{*}\in zer(A+B+C)$.
By (\ref{eq3-9}), we have $Cx_{B}^{k}\rightarrow Cx^{*}$. Notice that $x_{B}^{k_{n}}\rightharpoonup y$, since C is maximally monotone, it follows from the weak-to-strong sequentical closedness of $C$ that $ Cy=Cx^{*}$. Then $Cx_{B}^{k}\rightarrow Cy$.

We deduce from (\ref{eq3-8}) that
$\|x_{A}^{k}-x_{B}^{k}\|=\|Ty^{k}-y^{k}\|\rightarrow 0$ as $k \rightarrow +\infty$.
Then, we have
$ x_{A}^{k_{n}}\rightharpoonup y$. Since $\|y^k - z^k\| = \alpha_k \|z^k - z^{k-1}\| \rightarrow 0$ as $k\rightarrow +\infty$, we obtain $y^k \rightharpoonup z^{*}$. Therefore, we get
$\mu_{B}^{k_{n}}\rightharpoonup \frac {1}{\gamma}(z^{*}-y)$ and $\mu_{A}^{k_{n}}\rightharpoonup \frac {1}{\gamma}(y-z^{*}-\gamma Cy)$.

By Corollary 26.8 of \cite{bauschkebook2017}, we obtain
\begin{equation}\label{eq3-11}
\frac {1}{\gamma}(z^{*}-y)\in By \Rightarrow y = (I+\gamma B)^{-1}z^{*} = J_{\gamma B}z^{*},
\end{equation}
and $ y\in zer(A+B+C)$.
Therefore $y$ is the unique weak sequential cluster point of $\{x_{B}^{k}\}$. Then $\{x_{B}^{k}\}$ converges weakly to $J_{\gamma B}z^{*} \in zer(A+B+C)$.

\noindent (iii)\, The conclusion of $x_{A}^{k}\rightharpoonup J_{\gamma B}z^{*}$ comes from the fact that $\|x_{A}^{k}-x_{B}^{k}\|\rightarrow 0$
as $k \rightarrow  +\infty$ and $x_{B}^{k} \rightharpoonup J_{\gamma B}z^{*}$.

\noindent (iv)\, Let $ x^{*}=J_{\gamma B}z^{*}$, then $x^{*}\in zer(A+B+C)$. Let $\mu_{B}^{*}=\frac {1}{\gamma}(z^{*}-x^{*})\in Bx^{*}$, hence, we have $\mu_{B}^{*} \in Bx^{*}$. Let $\mu_{A}^{*}=\frac {1}{\gamma}(x^{*}-z^{*})-Cx^{*} \in Ax^{*}$, we derive from $x^{*}\in zer(A+B+C)$ and $\mu_{B}^{*} \in Bx^{*}$ that $\mu_{A}^{*} \in Ax^{*}$. It follows from that $B+C$ is monotone and $(x_{B}^{k}, \mu_{B}^{k})\in \textrm{ gra } B$, we have
$\langle x_{B}^{k}-x^{*}, \mu_{B}^{k}+ Cx_{B}^{k}-(\mu_{B}^{*}+Cx_{B}^{k})\rangle \geq 0$.

(a)\, Let $M=\{x^{*}\}\bigcup \{x_{A}^{k}| k\geq 0\}$. Then because $A$ is uniformly monotone on every nonempty bounded subset of $\textrm{ dom } A$. So there exists an increasing function $\phi_{A}: R_{+}\rightarrow [0, +\infty]$ that only vanishes at $0$ such that
\begin{align}\label{eq3-12}
\gamma \phi_{A}(\|x_{A}^{k}-x^{*}\|) &\leq \gamma \langle x_{A}^{k}-x^{*}, \mu_{A}^{k}-\mu_{A}^{*}\rangle
                     +\gamma \langle x_{B}^{k}-x^{*}, \mu_{B}^{k} +Cx_{B}^{k}-(\mu_{B}^{*}+Cx^{*})\rangle \nonumber \\
                     &=\gamma \langle x_{A}^{k}-x_{B}^{k}, \mu_{A}^{k}-\mu_{A}^{*}\rangle +\gamma \langle x_{B}^{k}-x^{*}, \mu_{A}^{k} -\mu_{A}^{*}\rangle \nonumber \\
                     \quad &+\gamma \langle x_{B}^{k}-x^{*}, \mu_{B}^{k} +Cx_{B}^{k}-(\mu_{B}^{*}+Cx^{*})\rangle \nonumber \\
                     &=\gamma \langle x_{A}^{k}-x_{B}^{k}, \mu_{A}^{k}-\mu_{A}^{*}\rangle +\gamma \langle x_{B}^{k}-x^{*}, \mu_{A}^{k} +\mu_{B}^{k}+Cx_{B}^{k}\rangle \nonumber \\
                     &=\langle x_{B}^{k}-x_{A}^{k}, x_{B}^{k}-\gamma\mu_{A}^{k}-(x^{*}-\gamma \mu_{A}^{*}) \rangle \nonumber \\
            &\leq \langle x_{B}^{k}-x_{A}^{k}, y^{k}-z^{*}\rangle+\gamma \langle x_{B}^{k}-x_{A}^{k}, Cx_{B}^{k}-Cx^{*}\rangle.
\end{align}
Since $y^{k}\rightharpoonup z^{*}$, $\|x_{B}^{k}-x_{A}^{k}\|\rightarrow 0$, and $Cx_{B}^{k}-Cx^{*}\rightarrow 0$, as $k\rightarrow +\infty$. Then, we obtain from the above inequality that $x_{A}^{k} \rightarrow x^{*}$.
Moreover, $x_{B}^{k}\rightarrow x^{*}$ because $x_{A}^{k}-x_{B}^{k}\rightarrow 0$, as $k\rightarrow +\infty$.

 (b)\, Since $A+C$ is monotone, we have
\begin{equation}\label{eq3-13}
0 \leq \langle x_{A}^{k} - x^{*}, \mu_{A}^{k} + Cx_{A}^{k} - (\mu_{A}^{*}+Cx^{*})   \rangle.
\end{equation}
Let $M=\{x^{*}\} \bigcup \{x_{B}^{k}| k\geq 0\}$. It follows from the uniformly monotone $B$ that
\begin{equation}\label{eq3-14}
\phi_{B}(\| x_{B}^{k} - x^{*} \|) \leq \langle x_{B}^{k} - x^{*}, \mu_{B}^{k} - \mu_{B}^{*} \rangle,
\end{equation}
where $\phi_{B}: R_{+}\rightarrow [0, +\infty]$ that only vanishes at $0$. Multiply $\gamma >0$ on the both side of (\ref{eq3-13}) and (\ref{eq3-14}), then we obtain
\begin{align}\label{eq3-15}
\gamma \phi_{B}(\| x_{B}^{k} - x^{*} \|) & \leq \gamma \langle x_{B}^{k} - x^{*}, \mu_{B}^{k} - \mu_{B}^{*} \rangle + \gamma \langle x_{A}^{k} - x^{*}, \mu_{A}^{k} + Cx_{A}^{k} - (\mu_{A}^{*}+Cx^{*})   \rangle \nonumber \\
& = \gamma \langle x_{B}^{k} - x_{A}^{k}, \mu_{B}^{k} - \mu_{B}^{*} \rangle + \langle x_{A}^{k} - x^{*}, \mu_{B}^{k} - \mu_{B}^{*} \rangle  \nonumber \\
& \quad + \gamma \langle x_{A}^{k} - x^{*}, \mu_{A}^{k} + Cx_{A}^{k} - (\mu_{A}^{*}+Cx^{*})   \rangle \nonumber \\
& = \gamma \langle x_{B}^{k} - x_{A}^{k}, \mu_{B}^{k} - \mu_{B}^{*} \rangle + \gamma \langle x_{A}^{k} - x^{*}, \mu_{A}^{k} + \mu_{B}^{k} + C x_{A}^{k}  \rangle \nonumber \\
& = \gamma \langle x_{B}^{k} - x_{A}^{k}, \mu_{B}^{k} - \mu_{B}^{*} \rangle + \langle x_{A}^{k} - x^{*}, x_{B}^{k} - x_{A}^{k} \rangle \nonumber \\
& \quad + \gamma \langle  x_{A}^{k} - x^{*}, Cx_{A}^{k} - Cx_{B}^{k}  \rangle.
\end{align}
Since $\lim_{k\rightarrow +\infty}\|x_{B}^{k} - x_{A}^{k}\| = 0$ and $C$ is $1/\beta$-Lipschitz continuous, then $\|Cx_{A}^{k} - Cx_{B}^{k}\| \rightarrow 0 $ as $k\rightarrow +\infty$. Hence, we derive from (\ref{eq3-15}) that $x_{B}^{k} \rightarrow x^{*}$ as $k\rightarrow +\infty$.

 (c)\, Because $Cx_{B}^{k} \rightarrow Cx^{*}$ and $x_{B}^{k} \rightharpoonup x^{*}$, thus $x_{B}^{k} \rightarrow x^{*}$ by the demiregularity of $C$. This completes the proof.
\end{proof}

\begin{theorem}\label{main-theorem2}
Let $H$ be a real Hilbert space. Let $A:H \rightarrow 2^{H}$ and $B:H\rightarrow 2^H$ be two maximally monotone operators. Let $C:H\rightarrow H$ be a $\beta$-cocoercive operator, for some $\beta> 0$. Let the operator $T$ be defined as (\ref{def-T}).
Let the iterative sequences $\{z^{k}\}, \{x_{A}^{k}\}, \{x_{B}^{k}\}$ are generated by Algorithm \ref{alg1}. Assume that
 $0 \leq \alpha_{k} \leq \alpha \leq 1 $, $ 0< \underline{\lambda} \leq \lambda_{k} < 1 $ and $\sum_{k=0}^{+\infty}\alpha_{k}\|z^{k}-z^{k-1}\|^{2}< +\infty$.
Then the following hold

\noindent \emph{(i)} $\{z^{k}\}$ converges weakly to a fixed point of $T$.

\noindent \emph{(ii)} Let $\lambda_{k}\geq \underline{\lambda}> 0 $ and $z^{*}$ be a fixed point of $T$. Then $\{x_{B}^{k}\}$ converges weakly to $ J_{\gamma B}z^{*}\in zer(A+B+C)$.

\noindent \emph{(iii)} Let $\lambda_{k}\geq \underline{\lambda}> 0 $ and $z^{*}$ be a fixed point of $T$. Then $\{x_{A}^{k}\}$ converges weakly to $ J_{\gamma B}z^{*}\in zer(A+B+C)$.

\noindent \emph{(iv)} Let $\lambda_{k}\geq \underline{\lambda}> 0 $ and $z^{*}$ be a fixed point of $T$. Suppose that one of the following conditions hold

\noindent \emph{(a)} $A$ is uniformly monotone on every nonempty bounded subset of $\textrm{ dom } A$.\\
\noindent \emph{(b)} $B$ is uniformly monotone on every nonempty bounded subset of $\textrm{ dom } B$.\\
\noindent \emph{(c)} $C$ is demiregular at every point $x \in zer(A+B+C)$. \\
Then $\{x_{A}^{k}\}$ and $\{x_{B}^{k}\}$ converge strongly to $J_{\gamma B}z^{*}\in zer(A+B+C)$.
\end{theorem}

\begin{proof}
(i)\, Since the iterative sequences $\{z^k\}$ generated by Algorithm \ref{alg1} can be rewritten as (\ref{eq3-1}). Hence, by Lemma \ref{iKM-convergence-2}, we can conclude that $\{z^k\}$ converges weakly to a fixed point of $T$.

On the other other hand, we deduce from $\sum_{k=0}^{+\infty}\alpha_k \|z^k - z^{k-1}\|^2 < +\infty$ that $\lim_{k\rightarrow +\infty}\alpha_k \|z^k - z^{k-1}\|=0$. Then we have
\begin{equation}
y^k - z^k = \alpha_k (z^k - z^{k-1}) \rightarrow 0, \, \textrm{ as }\, k \rightarrow +\infty.
\end{equation}
The conclusions of (ii)-(iv) follow the same proof of Theorem \ref{main-theorem}, so we omit it here.

\end{proof}

\begin{remark}
The conditions of inertial parameters $\{\alpha_k\}$ and $\{\lambda_k\}$ are different in Theorem \ref{main-theorem} and Theorem \ref{main-theorem2}. In Theorem \ref{main-theorem}, the relaxation parameters $\{\lambda_k\}$ are restricted by the upper bound of $\{\alpha_k\}$, while the relaxation parameters $\{\lambda_k\}$ are independent to the inertial parameters in Theorem \ref{main-theorem2}.

On the other hand, the inertial parameters $\{\alpha_k\}$ in Theorem \ref{main-theorem2} are depended on the iterative sequences. In general, it is recommended to set as follows:
$$
\alpha_k = \min \{ \frac{1}{k^2 \|z^k - z^{k-1}\|^2}, \overline{\alpha} \}, \quad \textrm{ where }\, \overline{\alpha}\in (0,1).
$$
However, it is seldom to see any numerical experiments reported on this particular choice of inertial parameters. We will present numerical experiments to demonstrate the effectiveness and efficiency of it.

\end{remark}

As applications, we consider the following general convex optimization problem:
\begin{equation}\label{sum-three-convex}
\min_{x\in H}\, h(x) + f(x) + g(x),
\end{equation}
where $f:H\rightarrow (-\infty,+\infty]$ and $g:H\rightarrow (-\infty,+\infty]$ are proper, lower semi-continuous convex functions, and $h:H\rightarrow R$ is a convex differentiable function with a $1/\beta$-Lipschitz continuous gradient. Under some qualification conditions, the first-order optimality condition of (\ref{sum-three-convex}) is equivalent to the monotone inclusion problem (\ref{more two-monotone-inclusion}). Therefore,
we obtain a class of inertial three-operator splitting algorithm for solving the convex optimization problem of the sum of three convex functions (\ref{sum-three-convex}).

\begin{theorem}\label{coro1}
Let $H$ be a real Hilbert space. Let $f,g:H\rightarrow (-\infty,+\infty]$ are proper, lower semi-continuous convex functions. Let $h:H\rightarrow R$ is convex differentiable with a $1/\beta$-Lipschitz continuous gradient. Suppose that $zer(\partial f+\partial g+\nabla h)\neq \varnothing$.
Let $z^0, z^{-1}\in H$, and set
\begin{equation}\label{inertial-cp-algorithm}
\left\{
\begin{aligned}
y^{k}&=z^{k}+\alpha_{k}(z^{k}-z^{k-1});\\
x_{g}^{k}&= prox_{\gamma g}y^{k};\\
x_{f}^{k}&= prox_{\gamma f}(2x_{g}^{k}-y^{k}-\gamma \nabla h(x_{g}^{k}));\\
z^{k+1}&=y^{k}+\lambda _{k}(x_{f}^{k}-x_{g}^{k})x. \\
\end{aligned}
\right.
\end{equation}
Assume that the parameters $\alpha_{k}, \gamma $ and $\lambda_{k}$ satisfy the conditions of Theorem \ref{main-theorem}.
Then the following hold

\noindent \emph{(i)} $\{z^{k}\}$ converges weakly to a fixed point of $T$, where $T := I-prox_{\gamma g}+prox_{\gamma f}(2prox_{\gamma g}-I-\gamma \nabla h( prox_{\gamma g}))$.

\noindent \emph{(ii)} Let $\lambda_{k}\geq \underline{\lambda}> 0 $ and $z^{*}$ be a fixed point of $T$. Then $\{x_{g}^{k}\}$ converges weakly to $ prox_{\gamma g}z^{*}\in zer(\partial f+\partial g+\nabla h)$.

\noindent \emph{(iii)} Let $\lambda_{k}\geq \underline{\lambda}> 0 $ and $z^{*}$ be a fixed point of $T$. Then $\{x_{f}^{k}\}$ converges weakly to $ prox_{\gamma g}z^{*}\in zer(\partial f+\partial g+\nabla h)$.

\noindent \emph{(iv)} Let $\lambda_{k}\geq \underline{\lambda}> 0 $ and $z^{*}$ be a fixed point of $T$. Suppose that one of the followiing conditions hold

 \noindent \emph{(a)} $f$ is uniformly convex on every nonempty bounded subset of $\textrm{ dom } \partial f$.\\
 \noindent \emph{(b)} $g$ is uniformly convex on every nonempty bounded subset of $\textrm{ dom } \partial g$.\\
 \noindent \emph{(c)} $\partial h$ is demiregular at every point $x \in zer(\partial f+\partial g+\nabla h)$.

Then $x_{f}^{k}$ and $\{x_{g}^{k}\}$ converge strongly to $ prox_{\gamma g}z^{*}\in zer(\partial f+\partial g+\nabla h)$.

\end{theorem}

\begin{proof}
Let $A=\partial f, B=\partial g, C=\nabla h$. Since the subdifferential of a proper convex and lower semicontinuous function is a maximally monotone operator, then $\partial f$ and $\partial g$ are maximally monotone operators. On the other hand, by the Baillon-Haddad theorem, $\nabla h$ is $\beta$ cocoercive. Therefore, we can get the conclusions of Theorem \ref{coro1} from Theorem \ref{main-theorem} immediately.
\end{proof}

Similarly, we obtain the following convergence theorem of the iterative sequence (\ref{inertial-cp-algorithm}) from Theorem \ref{main-theorem2}.

\begin{theorem}\label{coro2}

Let $H$ be a real Hilbert space. Let $f$ and $g:H\rightarrow (-\infty,+\infty]$ are proper closed lower semi-continuous convex function. Let $h:H\rightarrow R$ is convex differentiable function with a $1/\beta$-Lipschitz continuous gradient. Suppose that $zer(\partial f+\partial g+\nabla h)\neq \varnothing$.
Let the iterative sequences $\{z^{k}\}, \{x_{g}^{k}\}$, and $\{x_{f}^{k}\}$ are generated by (\ref{inertial-cp-algorithm}).
Assume that the parameters $\alpha_{k}, \gamma $ and $\lambda_{k}$ satisfy the conditions of Theorem \ref{main-theorem2}.
Then the following hold:

\noindent \emph{(i)} $\{z^{k}\}$ converges weakly to a fixed point of $T$, where $T := I-prox_{\gamma g}+prox_{\gamma f}(2prox_{\gamma g}-I-\gamma \nabla h( prox_{\gamma g}))$.

\noindent \emph{(ii)} Let $\lambda_{k}\geq \underline{\lambda}> 0 $ and $z^{*}$ be a fixed point of $T$. Then $\{x_{g}^{k}\}$ converges weakly to $ prox_{\gamma g}z^{*}\in zer(\partial f+\partial g+\nabla h)$.

\noindent \emph{(iii)} Let $\lambda_{k}\geq \underline{\lambda}> 0 $ and $z^{*}$ be a fixed point of $T$. Then $\{x_{f}^{k}\}$ converges weakly to $ prox_{\gamma g}z^{*}\in zer(\partial f+\partial g+\nabla h)$.

\noindent \emph{(iv)} Let $\lambda_{k}\geq \underline{\lambda}> 0 $ and $z^{*}$ be a fixed point of $T$. Suppose that one of the following conditions hold

 \noindent \emph{(a)} $f$ is uniformly convex on every nonempty bounded subset of $\textrm{ dom } \partial f$.\\
 \noindent \emph{(b)} $g$ is uniformly convex on every nonempty bounded subset of $ \textrm{ dom } \partial g$.\\
 \noindent \emph{(c)} $\partial h$ is demiregular at every point $x \in zer(\partial f+\partial g+\nabla h)$.

Then $\{x_{f}^{k}\}$ and $\{x_{g}^{k}\}$ converge strongly to $ prox_{\gamma g}z^{*}\in zer(\partial f+\partial g+\nabla h)$.

\end{theorem}

%%%%%%%%%%%%%%%%%%%%%%%%%%%%%%%%%%%%%%%%%%%%%%%%%%%%%%%%%%%%%%%%%%%%%%%%%%%%%%%%%%%%%%%%%%%%%%%%%%%%%%%%%%%%%%%%%%%%%%%%%%%%%%%%%%
\section{Numerical experiments}
\vskip 3mm

In this section, we apply the proposed inertial three-operator splitting algorithm (Algorithm \ref{alg1}) to solve a constrained image inpainting problem. We also compare it with the original three-operator splitting (TOS) algorithm. Notice that the convergence of the proposed inertial three-operator splitting algorithm is obtained in Theorem \ref{main-theorem} and Theorem \ref{main-theorem2}. The difference between Theorem \ref{main-theorem} and Theorem \ref{main-theorem2} is the requirement of the inertial parameters $\{\alpha_{k}\}$ and the relaxation parameters $\{\lambda_{k}\}$. Therefore, we refer to the proposed algorithm by iTOS-1 and iTOS-2, respectively.
All the experiments are conducted in a laptop of Lenovo Intel (R) core (TM)i7-4712MQ CPU and 4 GB memory. We run the codes in MATLAB 2014A.

\subsection{Constrained image inpainting problem}

Let $u\in R^{m\times n}$ be a given image, where $\{u_{ij}\}_{(i,j)\in \Omega}$ are observed and the rest are missed. We consider the following contrained image inpainting problem:

\begin{equation}\label{constrained-image-inpainting}
\min_{x\in C}\, \frac{1}{2}\|P_{\Omega}(u)-P_{\Omega}(x)\|_{F}^{2}+ \mu \|x\|_{*},
\end{equation}
where $\|\cdot\|_{F}$ is the Frobenius norm,  $\|\cdot\|_{*}$ is the nuclear norm, $C$ is a nonempty closed convex set, and $\mu >0$ is the regularization parameter. Here, $P_{\Omega}$ is defined by

\[P_{\Omega}(u)=\left\{\begin{array}{ll}
u_{ij}, &\text{$(i,j)\in \Omega$},\\
0, &
\textrm{otherwise}.
\end{array}\right.\]

The nuclear norm has been widely used in image inpainting and matrix completion problem, which is a convex relaxation of low rank constraint. See, for example \cite{yangjfandyuan2013,tang2016-1,tangyc20171}. Here, we introduce a nonempty closed convex $C$ in (\ref{constrained-image-inpainting}), which provides an easy way to incorporate priori information. In particular, we choose $C$ as a nonnegative set, that is $C=\{x\in R^{m*n}| x_{ij}\geq 0\}$. By virtue of the indicator function $\delta_{C}(x)$, the constrained image inpainting problem (\ref{constrained-image-inpainting}) is equivalent to the following unconstrained image inpainting problem,
\begin{equation}\label{unconstrained-image-inpainting}
\min_{x}\, \frac{1}{2}\|P_{\Omega}(u)-P_{\Omega}(x)\|_{F}^{2}+\lambda\|x\|_{*}+\delta_{C}(x).
\end{equation}
It is obvious that the optimization problem (\ref{unconstrained-image-inpainting}) is a special case of the optimization problem of the sum of three convex functions (\ref{sum-three-convex}). In fact, let $h(x)=\frac{1}{2}\|P_{\Omega}(u)-P_{\Omega}(x)\|_{F}^{2}$,
$g(x)=\lambda\|x\|_{*}$, and $f(x)=\delta_{C}(x)$. Then $h(x)$ is convex differentiable and $\nabla h(x)=P_{\Omega}(u)-P_{\Omega}(x)$ with $1$-Lipschitz continuous. The proximity operator of $g(x)$ can be computed by the singular value decomposition (SVD). See, for example \cite{Caijf2010SIAM}. And the proximity operator of $f(x)$ is the orthogonal projection onto the closed convex set $C$. Therefore, the three-operator splitting algorithm and the inertial three-operator splitting algorithm can be employed to solve the optimization problem (\ref{unconstrained-image-inpainting}).

\subsection{Evaluation and parameters setting}

To evaluate the quality of the restored images, we use the signal-to-noise ration (SNR) and the structural similarity
index (SSIM) \cite{wangzhou}, which are defined by
\begin{equation}
SNR = 20log\frac{\|x\|_{F}}{\|x-x_{r}\|_{F}},
\end{equation}
and
\begin{equation}
SSIM=\frac{(2u_{x}u_{x_r}+c_{1})(2\sigma_{xx_r}+c_{2})}{(u_{x}^{2}+u_{x_r}^{2}+c_{1})(\sigma_{x}^{2}+\sigma_{x_r}^{2}+c_{2})},
\end{equation}
where $x$ is the original image, $x_{r}$ is the restored image, $u_{x}$ and $u_{x_r}$ are the mean values of the original image $x$ and restored image $x_{r}$, respectively, $\sigma_{x}^{2}$ and $\sigma_{x_r}^{2}$ are the variances, $\sigma_{xx_r}^{2}$ is the covariance of two images, $c_{1}=(K_{1}L)^{2}$ and $c_{2}=(K_{2}L)^{2}$ with $K_{1}=0.01$ and $K_{2}=0.03$, and $L$ is the dynamic range of pixel values. $SSIM$ ranges from $0$ to $1$, and $1$ means perfect recovery.

We define the relative change between two successive iterative sequences as the stopping criterion, that is
\begin{equation}
\frac{\|z^{k+1}-z^{k}\|_{F}}{\|z^{k}\|_{F}}\leq \varepsilon,
\end{equation}
where $\varepsilon$ is a given small constant.

The selection of relaxation parameter $\lambda_k$ and step size $\gamma$ are crucial to the convergence speed of the three-operator splitting algorithm and the proposed inertial three-operator splitting algorithm. For the sake of fair comparison, we define these parameters in Table \ref{parameters-selection}.

\begin{table}[htbp]
\footnotesize
\centering
\caption{Parameters selection of the studied iterative algorithms.}
\vskip 2mm
\begin{tabular}{c|c|c|c|c}
\hline
Type & Methods  & Step size &  Relaxation parameter & Inertial parameter   \\
\hline
 \hline
\multirow{3}[1]{*}{Case 1}  & TOS  & \multirow{3}[1]{*}{$\gamma = 1.8$} & $\lambda_k =1$ & None  \\
& iTOS-1 & & $\lambda_k = 0.8$ & $\alpha_k = 0.2$  \\
& iTOS-2 & & $\lambda_k =1$ & $\alpha_k = \min \{ \frac{1}{k^2 \| z^k - z^{k-1} \|^2}, 0.2 \}$  \\
\hline
\multirow{3}[1]{*}{Case 2}  & TOS  & \multirow{3}[1]{*}{$\gamma = 1$} & $\lambda_k = 0.3$ & None  \\
& iTOS-1 & & $\lambda_k = 0.3$ & $\alpha_k = 0.5$  \\
& iTOS-2 & & $\lambda_k = 0.3$ & $\alpha_k = \min \{ \frac{1}{k^2 \| z^k - z^{k-1} \|^2}, 0.5 \}$  \\
\hline
\multirow{3}[1]{*}{Case 3}  & TOS  & \multirow{3}[1]{*}{$\gamma = 0.5$} & $\lambda_k = 1.75$ & None  \\
& iTOS-1 & & $\lambda_k = 1.4$ & $\alpha_k = 0.1$  \\
& iTOS-2 & & $\lambda_k = 1.75$ & $\alpha_k = \min \{ \frac{1}{k^2 \| z^k - z^{k-1} \|^2}, 0.1 \}$  \\
\hline
\end{tabular}\label{parameters-selection}
\end{table}

\subsection{Results and discussions}

The tested image is chosen from \cite{davis2015old}, which is a gray image of a building. See Figure \ref{tested-images}. We randomly remove some pixels from the original image with missing rate $40 \%$, $60 \%$ and $80 \%$, respectively. In each case, a random Gaussian noise with mean zero and standard variance $0.01$ and $0.05$ is added. The regularization parameter $\mu$ is tuned under different noise level. We set $\mu = 0.5$ for noise level $0.01$ and $\mu = 1.8$ for noise level $0.05$, respectively.

\begin{figure}[htbp]
     \setlength{\abovecaptionskip}{-5pt}
  \centering \scalebox{0.5}{\includegraphics{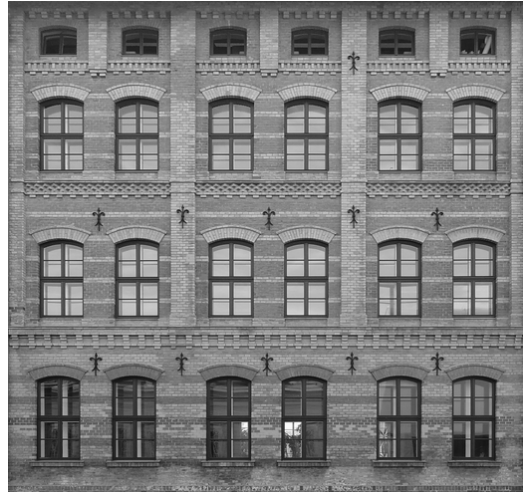}}
  \caption[]{Test image.}%
{\label{tested-images}}
\end{figure}

We test the performance of the studied iterative algorithms including TOS, iTOS-1 and iTOS-2 with parameters selection case $1$, $2$ and $3$ in Table \ref{parameters-selection}. The obtained results are presented in Table \ref{results1}, Table \ref{results2} and Table \ref{results3}, respectively.
We observe from Table \ref{results1} that when the stopping criterion $\varepsilon = 10^{-3}$, the TOS, iTOS-1 and iTOS-2 perform nearly the same in terms of the SSIM and the number of iteratons. The SNR of iTOS-1 and iTOS-2 are slightly higher than the TOS. From Table \ref{results2}, we can see that the iTOS-1 spends less iteration numbers than the TOS and iTOS-2. Further, we can see from Table \ref{results3} that the performance of the TOS and iTOS-2 are almost the same. In Table \ref{results3}, the iTOS-1 is the slowest.

We make the conclusion that from Table \ref{results1} to Table \ref{results3}, the iTOS-1 is faster than the original TOS for the same selection of the relaxation parameters, which are smaller than one. The iTOS-1 performs more stable than the iTOS-2. For visual display of the recovered images, we present them in Figure \ref{restored1}, Figure \ref{restored2} and Figure \ref{restored3}.

\begin{table}[htbp]
\footnotesize
\centering
\caption{Comparison results of parameters selection case $1$ in terms of SNR, SSIM and the number of iterations $k$.}
\vskip 2mm
\begin{tabular}{c|c|c|cccccccc}
\hline
\multirow{2}[1]{*}{Missing} & \multirow{2}[1]{*}{Noise}  & \multirow{2}[1]{*}{Methods} &   \multicolumn{3}{c}{$ \epsilon = 10^{-3}$} &  & \multicolumn{3}{c}{$ \epsilon = 10^{-5}$}
\\ \cline{4-6} \cline{8-10}
rate & level & &  $SNR$ & $SSIM$ & $k$  & & $SNR$ & $SSIM$ & $k$   \\
\hline
 \hline
\multirow{6}[1]{*}{$40 \%$}  & \multirow{3}[1]{*}{$0.01$} &  TOS  & $23.4856$ & $0.8877$ & $30$ & & $23.5961$ & $0.8895$ & $51$ \\
& & ITOS-1 & $23.4863$ & $0.8878$ & $30$ & & $23.5962$ & $0.8895$ & $51$ \\
& & ITOS-2 & $23.4868$ & $0.8878$ &  $30$ & & $23.5971$ & $0.8896$ & $980$ \\
\cline{2-10}
& \multirow{3}[1]{*}{$0.05$} &  TOS & $19.7318$ & $0.7639$ & $15$ & & $19.7395$ & $0.7642$ & $23$ \\
& & ITOS-1 & $19.7247$ & $0.7636$ & $13$ & & $19.7394$ & $0.7642$ & $20$ \\
& & ITOS-2 & $19.7326$ & $0.7639$ &  $15$ & & $19.7409$ & $0.7642$ & $39$ \\
\hline
\multirow{6}[1]{*}{$60 \%$}  & \multirow{3}[1]{*}{$0.01$} &  TOS & $20.6602$ & $0.8069$ & $49$ & & $20.8415$ & $0.8123$ & $85$ \\
& & ITOS-1 & $20.6622$ & $0.8070$ & $49$ & & $20.8415$ & $0.8123$ & $85$ \\
& & ITOS-2 & $20.6609$ & $0.8070$ &  $49$ & & $20.8420$ & $0.8123$ & $82$ \\
\cline{2-10}
& \multirow{3}[1]{*}{$0.05$} &  TOS & $17.5956$ & $0.6628$ & $21$ & & $17.6530$ & $0.6657$ & $36$ \\
& & ITOS-1 & $17.6035$ & $0.6632$ & $21$ & & $17.6529$ & $0.6657$ & $35$ \\
& & ITOS-2 & $17.5972$ & $0.6629$ &  $21$ & & $17.6530$ & $0.6657$ & $33$ \\
\hline
\multirow{6}[1]{*}{$80 \%$}  & \multirow{3}[1]{*}{$0.01$} &  TOS & $16.9807$ & $0.6433$ & $97$ & & $17.3557$ & $0.6617$ & $181$ \\
& & ITOS-1 & $16.9810$ & $0.6433$ & $97$ & & $17.3558$ & $0.6617$ & $181$ \\
& & ITOS-2 & $16.9813$ & $0.6434$ &  $97$ & & $17.3565$ & $0.6618$ & $176$ \\
\cline{2-10}
& \multirow{3}[1]{*}{$0.05$} &  TOS & $12.9316$ & $0.3830$ & $44$ & & $13.2010$ & $0.3986$ & $104$ \\
& & ITOS-1 & $12.9417$ & $0.3835$ & $44$ & & $13.2012$ & $0.3986$ & $104$ \\
& & ITOS-2 & $12.9332$ & $0.3830$ &  $44$ & & $13.2016$ & $0.3986$ & $98$ \\
\hline

\end{tabular}\label{results1}
\end{table}

\begin{table}[htbp]
\footnotesize
\centering
\caption{Comparison results of parameters selection case $2$ in terms of SNR, SSIM and the number of iterations $k$.}
\vskip 2mm
\begin{tabular}{c|c|c|cccccccc}
\hline
\multirow{2}[1]{*}{Missing} & \multirow{2}[1]{*}{Noise}  & \multirow{2}[1]{*}{Methods} &   \multicolumn{3}{c}{$ \epsilon = 10^{-3}$} &  & \multicolumn{3}{c}{$ \epsilon = 10^{-5}$}
\\ \cline{4-6} \cline{8-10}
rate & level & &  $SNR$ & $SSIM$ & $k$  & & $SNR$ & $SSIM$ & $k$   \\
\hline
 \hline
\multirow{6}[1]{*}{$40 \%$}  & \multirow{3}[1]{*}{$0.01$} &  TOS & $22.7328$ & $0.8733$ & $133$ & & $23.5898$ & $0.8894$ & $249$ \\
& & ITOS-1 & $23.2498$ & $0.8835$ & $76$ & & $23.5940$ & $0.8895$ & $134$ \\
& & ITOS-2 & $22.7345$ & $0.8734$ &  $133$ & & $23.5939$ & $0.8895$ & $244$ \\
\cline{2-10}
& \multirow{3}[1]{*}{$0.05$} &  TOS  & $19.5628$ & $0.7582$ & $55$ & & $19.7378$ & $0.7641$ & $103$ \\
& & ITOS-1 & $19.6817$ & $0.7621$ & $32$ & & $19.7389$ & $0.7641$ & $52$ \\
& & ITOS-2 & $19.5635$ & $0.7583$ &  $55$ & & $19.7389$ & $0.7641$ & $95$ \\
\hline
\multirow{6}[1]{*}{$60 \%$}  & \multirow{3}[1]{*}{$0.01$} &  TOS & $19.5084$ & $0.7689$ & $216$ & & $20.8319$ & $0.8120$ & $416$ \\
& & ITOS-1 & $20.2781$ & $0.7949$ & $124$ & & $20.8379$ & $0.8122$ & $224$ \\
& & ITOS-2 & $19.5102$ & $0.7689$ &  $216$ & & $20.8380$ & $0.8122$ & $421$ \\
\cline{2-10}
& \multirow{3}[1]{*}{$0.05$} &  TOS & $17.2028$ & $0.6438$ & $94$ & & $17.6490$ & $0.6655$ & $187$ \\
& & ITOS-1 & $17.4596$ & $0.6561$ & $55$ & & $17.6516$ & $0.6657$ & $99$ \\
& & ITOS-2 & $17.2036$ & $0.6439$ &  $94$ & & $17.6515$ & $0.6657$ & $179$ \\
\hline
\multirow{6}[1]{*}{$80 \%$}  & \multirow{3}[1]{*}{$0.01$} &  TOS & $14.3301$ & $0.5126$ & $375$ & & $17.3369$ & $0.6608$ & $877$ \\
& & ITOS-1 & $16.0750$ & $0.5979$ & $232$ & & $17.3485$ & $0.6614$ & $477$ \\
& & ITOS-2 & $14.3323$ & $0.5127$ &  $375$ & & $17.3415$ & $0.6611$ & $887$ \\
\cline{2-10}
& \multirow{3}[1]{*}{$0.05$} &  TOS & $11.6758$ & $0.3200$ & $184$ & & $13.1859$ & $0.3977$ & $547$ \\
& & ITOS-1 & $12.3940$ & $0.3544$ & $119$ & & $13.1953$ & $0.3982$ & $300$ \\
& & ITOS-2 & $11.6767$ & $0.3200$ &  $184$ & & $13.1953$ & $0.3982$ & $563$ \\
\hline

\end{tabular}\label{results2}
\end{table}

\begin{table}[htbp]
\footnotesize
\centering
\caption{Comparison results of parameters selection case $3$ in terms of SNR, SSIM and the number of iterations $k$.}
\vskip 2mm
\begin{tabular}{c|c|c|cccccccc}
\hline
\multirow{2}[1]{*}{Missing} & \multirow{2}[1]{*}{Noise}  & \multirow{2}[1]{*}{Methods} &   \multicolumn{3}{c}{$ \epsilon = 10^{-3}$} &  & \multicolumn{3}{c}{$ \epsilon = 10^{-5}$}
\\ \cline{4-6} \cline{8-10}
rate & level & &  $SNR$ & $SSIM$ & $k$  & & $SNR$ & $SSIM$ & $k$   \\
\hline
 \hline
\multirow{6}[1]{*}{$40 \%$}  & \multirow{3}[1]{*}{$0.01$} &  TOS & $23.3412$ & $0.8852$ & $55$ & & $23.5948$ & $0.8895$ & $95$ \\
& & ITOS-1 & $23.2780$ & $0.8841$ & $60$ & & $23.5946$ & $0.8895$ & $106$ \\
& & ITOS-2 & $23.3426$ & $0.8852$ &  $55$ & & $23.5951$ & $0.8895$ & $93$ \\
\cline{2-10}
& \multirow{3}[1]{*}{$0.05$} &  TOS & $19.6750$ & $0.7619$ & $23$ & & $19.7389$ & $0.7641$ & $38$ \\
& & ITOS-1 & $19.6579$ & $0.7613$ & $25$ & & $19.7389$ & $0.7641$ & $43$ \\
& & ITOS-2 & $19.6763$ & $0.7619$ &  $23$ & & $19.7395$ & $0.7642$ & $51$ \\
\hline
\multirow{6}[1]{*}{$60 \%$}  & \multirow{3}[1]{*}{$0.01$} &  TOS & $20.4405$ & $0.8001$ & $90$ & & $20.8395$ & $0.8122$ & $160$ \\
& & ITOS-1 & $20.3753$ & $0.7980$ & $99$ & & $20.8391$ & $0.8122$ & $178$ \\
& & ITOS-2 & $20.4415$ & $0.8001$ &  $90$ & & $20.8399$ & $0.8122$ & $157$ \\
\cline{2-10}
& \multirow{3}[1]{*}{$0.05$} &  TOS & $17.4874$ & $0.6575$ & $40$ & & $17.6520$ & $0.6657$ & $71$ \\
& & ITOS-1 & $17.4895$ & $0.6576$ & $45$ & & $17.6518$ & $0.6657$ & $79$ \\
& & ITOS-2 & $17.4884$ & $0.6575$ &  $40$ & & $17.6522$ & $0.6657$ & $69$ \\
\hline
\multirow{6}[1]{*}{$80 \%$}  & \multirow{3}[1]{*}{$0.01$} &  TOS & $16.5053$ & $0.6195$ & $173$ & & $17.3517$ & $0.6615$ & $341$ \\
& & ITOS-1 & $16.3921$ & $0.6138$ & $190$ & & $17.3508$ & $0.6615$ & $379$ \\
& & ITOS-2 & $16.5064$ & $0.6195$ &  $173$ & & $17.3525$ & $0.6616$ & $337$ \\
\cline{2-10}
& \multirow{3}[1]{*}{$0.05$} &  TOS & $12.6118$ & $0.3657$ & $92$ & & $13.1979$ & $0.3984$ & $217$ \\
& & ITOS-1 & $12.5646$ & $0.3632$ & $101$ & & $13.1972$ & $0.3983$ & $241$ \\
& & ITOS-2 & $12.6128$ & $0.3658$ &  $92$ & & $13.1985$ & $0.3984$ & $212$ \\
\hline

\end{tabular}\label{results3}
\end{table}

\begin{figure}[H]
    \centering
    \subfigure[Missing and noise image]{
        \scalebox{0.4}{\includegraphics{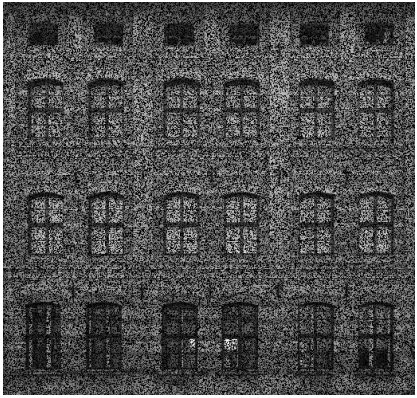}}
    }
    \subfigure[TOS]{
        \scalebox{0.4}{\includegraphics{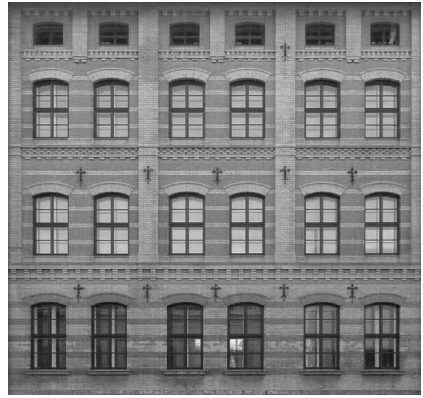}}
    }
    \\
    \subfigure[iTOS-1]{
        \scalebox{0.4}{\includegraphics{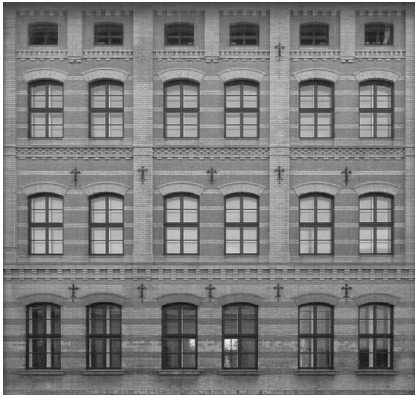}}
    }
    \subfigure[iTOS-2]{
        \scalebox{0.4}{\includegraphics{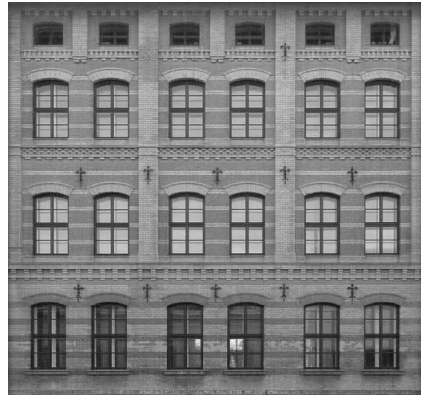}}
    }
    \caption{The missing and restored images. (a)\, $40 \%$ missing and $0.01$ noise image (SNR: 3.9784(dB), SSIM: 0.1384); (b)\, TOS (SNR: 23.5961(dB), SSIM: 0.8895); (c)\, iTOS-1 (SNR: 23.5962(dB), SSIM: 0.8895); (d)\, iTOS-2 (SNR: 23.5971(dB), SSIM: 0.8896).    }
    \label{restored1}
\end{figure}

\begin{figure}[H]
    \centering
    \subfigure[Missing and noise image]{
        \scalebox{0.4}{\includegraphics{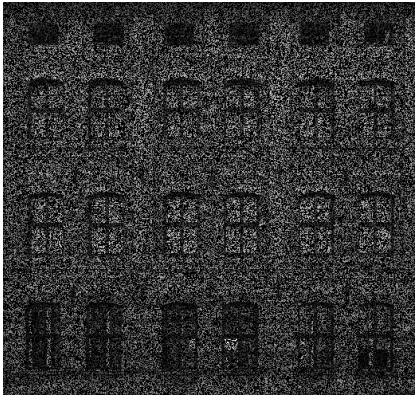}}
    }
    \subfigure[TOS]{
        \scalebox{0.4}{\includegraphics{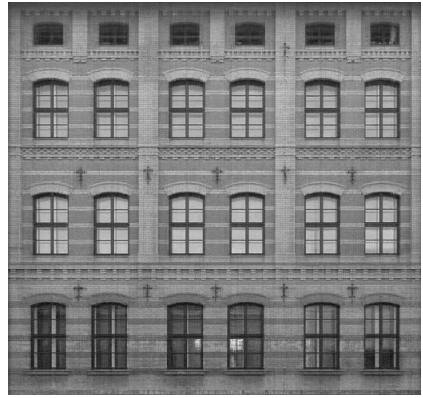}}
    }
    \\
    \subfigure[iTOS-1]{
        \scalebox{0.4}{\includegraphics{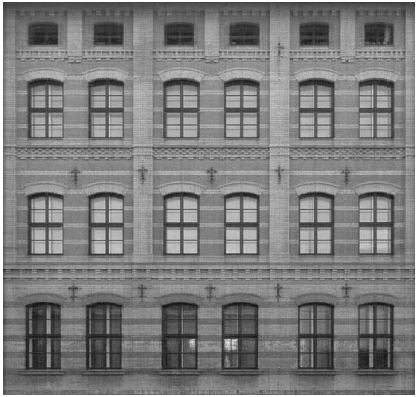}}
    }
    \subfigure[iTOS-2]{
        \scalebox{0.4}{\includegraphics{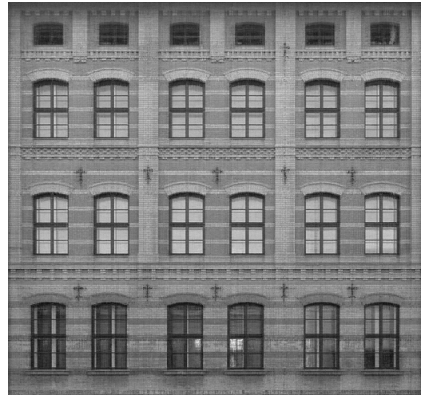}}
    }
    \caption{The missing and restored images. (a)\, $60 \%$ missing and $0.01$ noise image (SNR: 2.2314(dB), SSIM: 0.0788); (b)\, TOS (SNR: 20.8415(dB), SSIM: 0.8123); (c)\, iTOS-1 (SNR: 20.8415(dB), SSIM: 0.8123); (d)\, iTOS-2 (SNR: 20.8420(dB), SSIM: 0.8123).    }
    \label{restored2}
\end{figure}

\begin{figure}[H]
    \centering
    \subfigure[Missing and noise image]{
        \scalebox{0.4}{\includegraphics{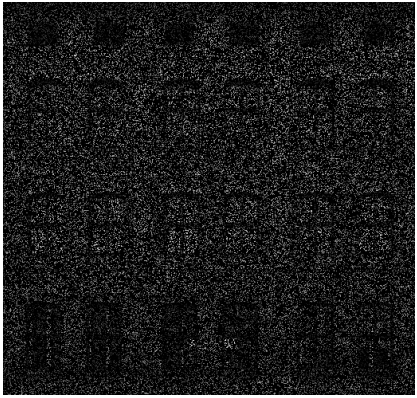}}
    }
    \subfigure[TOS]{
        \scalebox{0.4}{\includegraphics{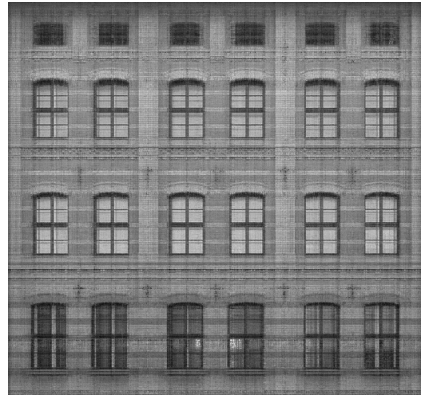}}
    }
    \\
    \subfigure[iTOS-1]{
        \scalebox{0.4}{\includegraphics{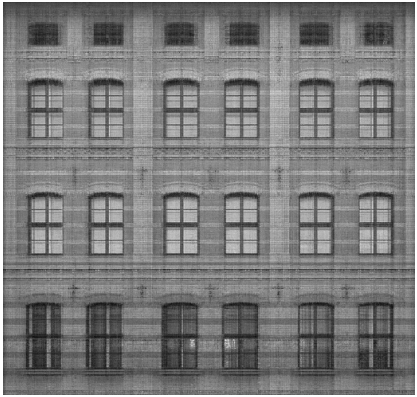}}
    }
    \subfigure[iTOS-2]{
        \scalebox{0.4}{\includegraphics{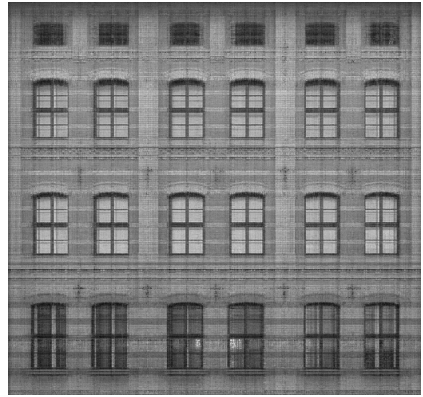}}
    }
    \caption{The missing and restored images. (a)\, $80 \%$ missing and $0.01$ noise image (SNR: 0.9641(dB), SSIM: 0.0336); (b)\, TOS (SNR: 17.3557(dB), SSIM: 0.6617); (c)\, iTOS-1 (SNR: 17.3558(dB), SSIM: 0.6617); (d)\, iTOS-2 (SNR: 17.3565(dB), SSIM: 0.6618).    }
    \label{restored3}
\end{figure}

%%%%%%%%%%%%%%%%%%%%%%%%%%%%%%%%%%%%%%%%%%%%%%%%%%%%%%%%%%%%%%%%%%%%%%%%%%%%%%%%%%%%%%%%%%%%%%%%%%%%%%%%%%%%%%%
\section{Conclusions and future works}

The three-operator splitting algorithm is a new operator splitting algorithm for solving the monotone inclusion problem (\ref{more two-monotone-inclusion}). In this paper, we proposed a class of inertial three-operator splitting algorithm, which combined the inertial methods and the three-operator splitting algorithm. We analyzed the convergence of the proposed algorithm based on the iKM iteration (\ref{inertial-KM-iteration}). As a direct application, we developed an inertial three-operator splitting algorithm for solving the convex optimization problem (\ref{sum-three-convex}), which had wide applications in signal and image processing. We also presented numerical experiments on the constrained image inpainting problem (\ref{constrained-image-inpainting}) to demonstrate the advantage of introducing the inertial terms.

When we finished this work, we find that \cite{attouch:hal-01708905,attouch:hal-01782016} also prove the convergence of the iKM iteration (\ref{inertial-KM-iteration}) under certain conditions on the inertial parameters and the relaxation parameters, which are different from Lemma \ref{iKM-convergence-1} and Lemma \ref{iKM-convergence-2}. It is natural to generalize the convergence analysis of the iKM iteration in \cite{attouch:hal-01708905,attouch:hal-01782016} to the proposed inertial three-operator splitting algorithm. We would like to compare the performance of variants inertial three-operator splitting algorithm for wide application problems, and study the convergence rate analysis of the inertial three-operator splitting algorithm in the context of convex minimization in the future works.

%%%%%%%%%%%%%%%%%%%%%%%%%%%%%%%%%%%%%%%%%%%%%%%%%%%%%%%%%%%%%%%%%%%%%%%%%%%%%%%%%%%%%%%%%%%%%%%%%%%%%%%%%%%%%%%%55
\section*{Acknowledgements}
We are appreciated for Professor Damek Davis for sharing his code online.

%%%%%%%%%%%%%%%%%%%%%%%%%%%%%%%%%%%%%%%%%%%%%%%%%%%%%%%%%%%%%%%%%%%%%%%%%%%%%%%%%%%%%%%%%%%%%%%%%%%%%%%%%%%%%%%%55
\section*{Funding}

This research was funded by the National Natural Science Foundations of China (11661056, 11771198, 11771347, 91730306, 41390454, 11401293), the China Postdoctoral Science Foundation (2015M571989)
and the Jiangxi Province Postdoctoral Science Foundation (2015KY51).

%%%%%%%%%%%%%%%%%%%%%%%%%%%%%%%%%%%%%%%%%%
\section*{Conflicts of interest}
The authors declare no conflict of interest.

% References
%\bibliographystyle{unsrt}
%\bibliography{klreference-en,essayfirst-en} %

\end{document}